\documentclass{birkjour}
\usepackage{amsfonts,amssymb,amscd,amsmath,enumerate,verbatim,calc,cite}
\usepackage[colorlinks,citecolor=blue]{hyperref}
 \theoremstyle{definition}
 
 \theoremstyle{remark}

 \numberwithin{equation}{section}

\newtheorem{theorem}{Theorem}[section]
\newtheorem{lemma}{Lemma}[section]

\theoremstyle{remark}

\theoremstyle{remark}

\theoremstyle{definition}

\newtheorem{defi}{Definition}[section]

\DeclareMathOperator{\tr}{trace}
\DeclareMathOperator*{\Ric}{Ric}
\DeclareMathOperator{\Span}{Span}

\newcommand{\co}{\nabla}

\newcommand{\p}{{\Pi_k}}
\newcommand{\nn}{(n_1,\ldots,n_k)}
\newcommand{\n}{(n_1,\ldots,n_k)}
\let\~\tilde
\newcommand{\x}{\mathfrak{X}}

\begin{document}

%
%
%
%
%
%
%
%
%

\title[Chen's inequality for C-totally real submanifolds in ...]
 {Chen's inequality for C-totally real submanifolds in a generalized  $(\kappa ,\mu)$-space form}

\author[M. Faghfouri]{Morteza Faghfouri}

\address{Department of Pure Mathematics,  Faculty of Mathematical Sciences, University of Tabriz, Tabriz, Iran.}

\email{faghfouri@tabrizu.ac.ir}

\author[N. Ghaffarzadeh]{Narges Ghaffarzadeh}
\address{Department of Pure Mathematics,  Faculty of Mathematical Sciences, University of Tabriz, Tabriz, Iran.}

\email{n.ghaffarzadeh@tabrizu.ac.ir}



\begin{abstract}
In this paper, we obtain a basic Chen's inequality for a C-totally real submanifold in a generalized $(\kappa ,\mu)$-contact space forms involving intrinsic invariants, namely the scalar curvature and the sectional curvatures of the submanifold on left hand side and the main extrinsic invariant, namely the squared mean curvature on the right hand side. Inequalities between the squared mean curvature, $k$-Ricci curvature and  Ricci curvature are also obtained.
\end{abstract}

\maketitle
\section{Introduction}

In the theory of submanifolds, the study  of immersibility of a Riemannian manifold in a Euclidean space is one of the fundamental problems.
According to the well-known theorem of J. Nash in 1956 \cite{NashJohn:TheImbeddingProblem}, every Riemannian manifold can be isometrically
embedded in some Euclidean spaces with sufficiently high codimension.

In \cite{chen:Somepinching}, B.-Y. Chen defined a Riemannian invariant $\delta_M=\tau-\inf K$ for any
Riemannian manifold M, where $\tau$ is the scalar curvature of $M$ and $(\inf K)(p)=\inf\{K(\pi)|$ plane sections $\pi \subset T_pM\}$. Also, in \cite{chen:Somepinching}, Chen obtained a necessary condition for the existence of minimal isometric immersion from a given Riemannian manifold into Euclidean space and established a sharp inequality for a submanifold in a real space form using the scalar curvature and the sectional curvature and squared mean curvature.
In \cite{Chen:RelationsBetweenRicci}, he gave a sharp relationship between the squared mean curvature and
the Ricci curvature for the submanifolds in a real space form.
This inequalities are also sharp, and many nice classes of submanifolds realize equality in inequalities.
In \cite{Chen:Somenewobstructionstominimal}, B. -Y. Chen introduced new types of curvature invariants, by defining two strings of scalar-valued Riemannian curvature functions, namely $\delta\nn$ and $\~\delta\nn.$ The first string of $\delta$-invariants, $\delta\nn$, extend naturally the Riemannian invariant introduced in \cite{chen:Somepinching}.

 Many papers  studied Chen invariants and inequalities, like complex space forms, cosymplectic space forms, warped product spaces and Sasakian space forms \cite{AlegreCarriazoKimYoon:chenInequalityforGeneralizedSpaceForm,faghfouriGhaffarzade:invariant,faghfouriGhaffarzade:doublywarped,faghfouri:ondoublyimmersion,TripathiHong:Riccicurvatureofsubmanifolds,Kim:Choi:ABasicInequalityCosymplectic,Ion:Riccicurvatureofsubmanifolds,tripathi:c-totallyrealwarped}.
In \cite{carriazo.Tripathi:Generalizedspaceforms},  A. Carriazo, V. Mart{\'{\i}}n Molina and M.M. Tripathi introduce generalized $(\kappa,\mu)$-space forms as an almost contact metric manifold $(\~M,\phi,\xi,\eta,\langle,\rangle)$ whose
curvature tensor can be written as
\begin{align*}
R=f_1R_1+f_2R_2+f_3R_3+f_4R_4+f_5R_5+f_6R_6,
\end{align*}
where $f_1, f_2, f_3,f_4,f_5,f_6$  are differentiable functions on $\~M$.
 generalized $(\kappa,\mu)$-space forms with divided $R_5$ whenever the curvature tensor  can be written as
\begin{align*}
R=f_1R_1+f_2R_2+f_3R_3+f_4R_4+f_{5,1}R_{5,1}+f_{5,2}R_{5,2}+f_6R_6,
\end{align*}
where $f_1, f_2, f_3,f_4,f_{5,1}, f_{5,2},f_6$  are differentiable functions on $\~M$. Obviously, any  generalized Sasakian $(\kappa,\mu)$-space form is  a generalized Sasakian $(\kappa,\mu)$-space form  with divided $R_5.$

M.M. Tripathi and J.S. Kim \cite[Theorem 5.2 ]{TripathiKin:C-totalyKMcontact} studied
the relationship between the scalar curvature, the  sectional curvature and   the squared mean curvature for $C$-totally real submanifolds in a $(\kappa,\mu)$-contact space forms.

   In this paper,  we improved Theorem 5.2 in \cite{TripathiKin:C-totalyKMcontact} for  a $C$-totally real submanifold of generalized $(\kappa,\mu)$-space forms with divided $R_5$. In  section 2, we recall some necessary details background on Riemannian invariant on Riemannian manifolds, contact metric manifold, $C$-totally real submanifolds and contact metric manifolds. In section 3, we establish a basic Chen's inequality for $C$-totally real submanifolds in  a generalized  $(\kappa,\mu)$-space form  with divided $R_5.$ Sections 4 and 5 contain an inequality between the squared mean curvature, $k$-Ricci curvature and Ricci curvature. Finely, In section 6, we apply these results to get corresponding results for $C$-totally real submanifolds in a generalized $(\kappa,\mu)$-contact space forms $\~M(f_1, \ldots ,f_6)$ with $f_3=f_1-1$.
\section{Preliminaries}
The Riemannian invariants are the intrinsic characteristics of a Riemannian manifold. The scalar curvature is the most studied scalar valued Riemannian invariant on Riemannian manifolds.

 Let $M$ be an $n$-dimensional Riemannian manifold.  We denote by $K(\pi)$ the sectional curvature of $M$ for a plane section $\pi$  in $T_p M$. For any orthonormal basis $\lbrace e_1,\ldots,e_n\rbrace$ for $T_p M$,
The scalar curvature
$\tau(p)$ of $M$ at $p$ is defined by
$\tau(p)=\sum_{1\leq i<j\leq n} K(e_i\wedge e_j),$
where $K(e_i\wedge e_j)$ is the sectional curvature of the plane section spanned by $e_i$ and $e_j$ at $p\in M$.

Let $\Pi_k$ be a $k$-plane section of $T_pM$ and $\{e_1,\ldots , e_k\}$
any orthonormal basis of $\Pi_k$. The scalar curvature $\tau(\Pi_k)$ of $\Pi_k$ is given by
\begin{align}\label{eq:11}
\tau(\Pi_k)=\sum_{1\leq i<j\leq k}K(e_i\wedge e_j).
\end{align}

The scalar curvature $\tau(p)$ of $M$ at $p$ is identical with the scalar curvature of
the tangent space $T_pM$ of $M$ at $p$, that is, $ \tau(p)=\tau(T_pM)$.
Also, we denote by $(\inf K)(p)=\inf\{K(\pi)|\pi\subset T_pM, \dim\pi=2\}$, and introduce the first Chen invariant $\delta_M(p)=\tau(p)-(\inf K)(p)$,
which is certainly an intrinsic character of $M$.

Suppose $\p$ is a $k$-plane section of $T_pM$ and $U$ a unit vector in $\p$. We choose an orthonormal basis $\{e_1, \ldots,e_k\}$ of $\p$ such that $e_1=U$.
The Ricci curvature $\Ric_{{\Pi_k}}$ of ${\Pi_k}$ at $U$ is given by
\begin{align}\label{eq:10}
\Ric_{{\Pi_k}}(U)=K_{12}+\cdots+K_{1k},
\end{align}
where $K_{ij}$ is the sectional curvature of the plane section spanned by $e_i$ and $e_j$.
The $\Ric_\p(U)$ is called a $k$-Ricci curvature. For each integer $k$, $2\leq k\leq n$, the Riemannian invariant $\theta_k$ on $n$-dimensional Riemannian Manifold $M$ is defined by
\begin{align}\label{eq:12}
\theta_k(p)=\left(\dfrac{1}{k-1}\right)\inf_{{\Pi_k},X}\Ric_{{\Pi_k}}(X), p\in M,
\end{align}
where ${\Pi_k}$ is $k$-plane sections in $T_pM$ and $X$ is unit vector in ${\Pi_k}$ \cite{Chen:RelationsBetweenRicci}.

For an integer $k\geq0$ denote by $\mathcal{S}(n, k)$ the finite set consisting of unordered $k$-tuples
$(n_1,\ldots,n_k)$ of integers $\geq2$ satisfying $n_1 < n$ and $n_1 +\cdots+n_k \leq n$. Denote
by $\mathcal{S}(n)$ the set of unordered $k$-tuples with $k\geq0$ for a fixed $n.$

For each $k$-tuple $\nn\in\mathcal{S}(n), \delta\nn(p)$ is defined by
\begin{align}
\delta\nn(p)=\tau-\inf\{\tau(L_1)+\cdots+\tau(L_k)\},\label{eq:deltainf}
\end{align}
where $L_1,\ldots,L_k$ run over all $k$ mutually orthogonal subspaces of $T_pM$ such that $\dim L_j=n_j, j=1,\ldots,k.$
Similarly, $\~\delta\nn(p)$ is defined by
\begin{align}
\~\delta\nn(p)=\tau-\sup\{\tau(L_1)+\cdots+\tau(L_k)\}.\label{eq:deltasup}
\end{align}
Obviously, $\delta(\emptyset)=\~\delta(\emptyset)=\tau$ for $k=0.$ It is also clear that
\[\delta\nn\geq\~\delta\nn\]
for any $k$-tuple $\nn\in\mathcal{S}(n).$

\begin{defi}[\cite{Chen:Somenewobstructionstominimal}]
A Riemannian $n$-manifold $M$ is called an $\mathcal{S}\nn$-space for a given $k$-tuple $\nn\in\mathcal{S}(n)$ if it satisfies
\(\delta\nn=\~\delta\nn,\)
identically.
\end{defi}
For each $\nn\in \mathcal{S}(n)$, let $c\nn$ and $b\nn$ denote
\begin{align}
\begin{split}
c\nn=&\dfrac{n^2(n+k-1-\sum_{j=1}^{k}n_j)}{2(n+k-\sum_{j=1}^kn_j)},\\
b\nn =&\dfrac{1}{2}n(n-1)-\dfrac{1}{2}\sum_{j=1}^k n_j(n_j-1).
\end{split}
\end{align}

Let $M$ be an $n$-dimensional submanifold in a manifold $\~M$ equipped with a Riemannian metric $\langle.,.\rangle$. The Gauss and Weingarten formulae are given respectively by
$\~\co_X Y=\co_X Y+\sigma(X,Y)$ and $\~\co_X N=-A_N X+\co^{\bot} _X N$ for all $X,Y\in\x(M)$ and $N\in T^{\bot}M$, where $\~\co$, $\co$ and $\co^{\bot}$ are Riemannian, induced Riemannian and induced normal connections in $\~M$, $M$ and the normal bundle
$T^{\bot}M$ of $M$ respectively, and $\sigma$ is the second fundamental form related to the shape operator $A_N$ in the direction of $N$ by $\langle\sigma(X,Y),N\rangle=\langle A_N X,Y\rangle$. Then, the Gauss equation is given by
\begin{align}\label{030}\begin{split}
\~R(X,Y,Z,W)=R(X,Y,Z,W)-\langle\sigma(X,W),\sigma(Y,Z)\rangle\\
+\langle\sigma(X,Z),\sigma(Y,W)\rangle\end{split}
\end{align}
for all $X,Y,Z,W\in\x(M)$, where $\~R$ and $R$ are the curvature tensors of $\~M$ and $M$ respectively. The mean curvature vector $H$ is expressed by $nH=\tr(\sigma)$.
The submanifold $M$ is totally geodesic in $\~M$ if $\sigma=0$, and minimal if $H=0$. If
$\sigma(X,Y)=\langle X,Y\rangle H$, for all $X,Y\in\x(M)$, then $M$ is totally umbilical.
The  null space of $M$ at a point $p\in M$ is  $\mathcal{N}_p=\{X\in T_pM \,|\, \sigma(X,Y)=0,\, \forall Y\in T_p(M)\}.$
A manifold $\~M$ is called an almost contact metric manifold if there is an almost contact metric structure $(\phi,\xi,\eta,\langle,\rangle)$ consisting of a $(1, 1)$ tensor field $\phi$, a vector field $\xi$, a $1$-form $\eta$ and a compatible Riemannian metric $\langle.,.\rangle$ satisfying
\begin{gather}
\phi^2=-I+\eta\otimes\xi, \eta(\xi)=1, \eta\circ\phi=0, \phi\xi=0, \\
\langle X,Y\rangle=\langle\phi X,\phi Y\rangle+\eta(X)\eta(Y),\\
\langle X,\xi\rangle=\eta(X),\langle X,\phi Y\rangle=-\langle\phi X,Y\rangle,
\end{gather}
for all $X,Y\in\x(\~M)$.
An almost contact metric structure becomes a contact metric structure if $d\eta=\Phi$, where $\Phi(X,Y)=\langle X,\phi Y\rangle$ is the fundamental $2$-form of $\~M$.

An almost contact metric structure of $\~M$ is said
to be normal if the Nijenhuis torsion $[\phi,\phi]=-2d\eta\otimes\xi$. A normal
contact metric manifold is called a Sasakian manifold. It can be proved that
an almost contact metric manifold is Sasakian if and only if
$(\co_X\phi)Y=\langle X,Y\rangle\xi-\eta(Y)X,$
for any $X, Y\in \x(\~M)$ or  equivalently, a contact metric structure is a Sasakian structure if and only if $\~R$ satisfies
$\~R(X,Y)\xi=\eta(Y)X-\eta(X)Y,$
for $X,Y\in\x(\~M)$.
In a contact metric manifold $\~M$, the $(1, 1)$-tensor field $h$ is  defined by $2h =\mathcal{L}_\xi\phi$, which is the Lie derivative of $\phi$ in the characteristic direction $\phi$. It is
symmetric and satisfies
\begin{gather*}
h\xi=0,\quad h\phi+\phi h=0, \quad  \tr(h)=\tr(\phi h)=0,
\~\co\xi=-\phi-\phi h,
\end{gather*}
where $\~\co$ is Levi-Civita connection.

Let $(M, \phi,\xi,\eta,\langle,\rangle)$ be an almost contact metric manifold. A $\phi$-section of $M$ at $p\in M$ is a section $\Pi\subset T_p\~M$ spanned by a unit vector $X_p$ orthogonal to $\xi_p$, and $\phi X_p$. The $\phi$-sectional curvature of $\Pi$ is defined by $\~K(X,\phi X)=\~R(X,\phi X,\phi X,X)$. A Sasakian manifold with constant $\phi$-sectional curvature $c$
is called a Sasakian space form and is denoted by
$\~M(c).$
A contact metric manifold $(\~M,\phi,\xi,\eta,\langle,\rangle)$ is said to be a $(\kappa,\mu)$-contact manifold if its curvature tensor satisfies the condition
$\~R(X,Y)\xi=\kappa(\eta(Y)X-\eta(X)Y)+\mu(\eta(Y)hX-\eta(X)hY),$
where  $\kappa$ and $\mu$ are real constant numbers.
If the $(\kappa,\mu)$-contact metric manifold $\~M$ has constant $\phi$-sectional curvature $c$, then it is said  to be a $(\kappa,\mu)$-contact space form.
\begin{defi}(\cite{carriazo.Tripathi:Generalizedspaceforms})
We say that an almost contact metric manifold $(\~M,\phi,\xi,\eta,\langle,\rangle)$ is a generalized $(\kappa,\mu)$-space form if there exist functions $f_1, f_2, f_3,f_4,f_5,f_6$ defined on $\~M$ such that
\begin{align}\label{029}
R=f_1R_1+f_2R_2+f_3R_3+f_4R_4+f_5R_5+f_6R_6,
\end{align}
where $R_1, R_2, R_3, R_4, R_5, R_6$ are the following tensors
\begin{align*}
&R_1(X,Y)Z=\langle Y,Z\rangle X-\langle X,Z\rangle Y,\\
&R_2(X,Y)Z=\langle X,\phi Z\rangle\phi Y-\langle Y,\phi Z\rangle\phi X+2\langle X,\phi Y\rangle\phi Z,\\
&R_3(X,Y)Z=\eta(X)\eta(Z)Y-\eta(Y)\eta(Z)X+\langle X,Z\rangle\eta(Y)\xi-\langle Y,Z\rangle\eta(X)\xi,\\
&R_4(X,Y)Z=\langle Y,Z\rangle hX-\langle X,Z\rangle hY+\langle hY,Z\rangle X-\langle hX,Z\rangle Y,\\
&R_5(X,Y)Z=\langle hY,Z\rangle hX-\langle hX,Z\rangle hY+\langle\phi hX, Z\rangle\phi hY-\langle\phi hY,Z\rangle\phi hX,\\
&R_6(X,Y)Z=\eta(X)\eta(Z)hY-\eta(Y)\eta(Z)hX+\langle hX,Z\rangle\eta(Y)\xi-\langle hY,Z\rangle\eta(X)\xi.
\end{align*}
for all vector fields $X, Y, Z $ on $\~M$, where $2h=\mathcal{L}_\xi \phi$ and $\mathcal{L}$  is the  Lie derivative. We will denote such a manifold by $\~M(f_1,\ldots,f_6)$.
\end{defi}
For example, $(\kappa,\mu)$-contact space forms are  generalized $(\kappa,\mu)$-space forms, with constant functions
$$f_1=\frac{c+3}{4}, f_2=\frac{c-1}{4}, f_3=\frac{c+3}{4}-\kappa, f_4=1, f_5=\frac{1}{2}, f_6=1-\mu. $$
And also, generalized Sasakian space forms $\~M(f_1,f_2,f_3)$  are generalized $(\kappa,\mu)$-space forms, with $f_4=f_5=f_6=0$ \cite{alegre.blair:GeneralizedSasakianspaceforms}.

\begin{defi}[\cite{carriazo:Generalizeddivided}]
An almost contact metric manifold $(\~M,\phi,\xi,\eta,\langle,\rangle)$ is a generalized $(\kappa,\mu)$-space form with divided $R_5$ if there exist function $f_1, f_2, f_3,$ $ f_4, f_{5,1}, f_{5,2}, f_6$ defined on $\~M$ such that
\begin{align}\label{028}
R=f_1R_1+f_2R_2+f_3R_3+f_4R_4+f_{5,1}R_{5,1}+f_{5,2}R_{5,2}+f_6R_6,
\end{align}
where $R_{5,1}, R_{5,2}$ are the following tensors
\begin{align*}
&R_{5,1}(X,Y)Z=\langle hY,Z\rangle hX-\langle hX,Z\rangle hY,\\
&R_{5,2}(X,Y)Z=\langle\phi hY, Z\rangle\phi hX-\langle\phi hX,Z\rangle\phi hY,
\end{align*}
for all vector fields $X, Y, Z $ on $\~M$.
\end{defi}
We will denote such a manifold by $\~M(f_1,f_2,f_3,$ $f_4,f_{5,1}, f_{5,2},f_6)$. It follows that $R_5 = R_{5,1}-R_{5,2}$.  It is obvious that, if $\~M(f_1,\ldots,f_6)$  is  a generalized $(\kappa,\mu)$-space form then $\~M$ is a generalized $(\kappa,\mu)$-space form with divided $R_5$ with $f_{5,1}=f_5$ and $f_{5,2}=-f_5$.
A non-Sasakian $(\kappa,\mu)$-space form is the generalized $(\kappa,\mu)$-space form with divided $R_5$ with $$f_1=\frac{2-\mu}{2}, f_2=-\frac{\mu}{2},f_3=\frac{2-\mu-2\kappa}{2},f_4=1, f_{5,1}=\frac{2-\mu}{2(1-\kappa)}, f_{5,2}=\frac{2\kappa-\mu}{2(1-\kappa)},$$ and $f_6=1-\mu$ but not the generalized $(\kappa,\mu)$-space form.

\begin{theorem}[\cite{carriazo.Tripathi:Generalizedspaceforms}]\label{21}
Let $ M(f_1, \ldots, f_6)$ be a generalized $(\kappa ,\mu)$-space form. If $M$ is a Sasakian manifold, then  $f_2=f_3=f_1-1.$
\end{theorem}
\begin{theorem}[\cite{carriazo.Tripathi:Generalizedspaceforms}]\label{21}
Let $ M(f_1, \ldots, f_6)$ be a generalized $(\kappa ,\mu)$-space form. If $M$ is a contact metric manifold with $f_3=f_1-1$, then it is a Sasakian manifold.
\end{theorem}

Let $M$ be a submanifold in a contact manifold $\tilde{M}$. $M$ is called a $C$-totally real submanifold if every tangent vector of $M$ belongs to the contact distribution \cite{yamaguchi:c-totallyrealsubmanifolds}.
Thus,  $M$  is a $C$-totally real submanifold if $\xi$ is normal to $M$. A submanifold $M$ in an almost contact metric
manifold  $\tilde{M}$ is called anti-invariant  if $\phi(TM)\subset T^\bot(M)$ \cite{yano:Antiinvariantsubmanifolds}. If a submanifold
$M$ in a contact metric manifold is normal to the structure vector field $\xi$,
then it is anti-invariant. Thus $C$-totally real submanifolds in a contact
metric manifold are anti-invariant, as they are normal to $\xi.$
For a $C$-totally real submanifold in a contact metric manifold we have
$\langle A_\xi X,Y\rangle=-\langle\~\co_X\xi,Y\rangle=\langle\phi X+\phi hX,Y\rangle,$ which implies that
\begin{align}\label{4.1}
A_\xi=(\phi h)^T,
\end{align}
where $(\phi h)^T$ is the tangential part of $\phi h X$ for all $X\in\x(M)$.

We state the following  Lemmas of Chen for later uses.
\begin{lemma}[B. Y. Chen \cite{chen:Somepinching}]\label{lem:cheninequality}
If $a_1,\ldots,a_n,a_{n+1}$ are $n+1 (n>1)$ real numbers such that
$$\dfrac{1}{n-1}(\sum_{i=1}^{n}a_i)^{2}=\sum_{i=1}^{n}a_{i}^{2}+a_{n+1},$$
then $2a_1a_2\geq a_{n+1}$, with equality holding if and only if $a_1+a_2=a_3=\ldots=a_n$.
\end{lemma}

\section{Certain basic inequalities}
In this section, We establish an inequality for submanifolds in a  generalized $(\kappa,\mu)$-contact space form with divided $R_5$,
$\~M(f_1, \ldots ,f_6)$ involving intrinsic invariants, namely the sectional curvature and the scalar curvature and the extrinsic invariant, namely the squared mean curvature.

\begin{lemma}
In an $n$-dimensional $C$-totally real submanifold $M$ in a $(2m+1)$-dimensional generalized $(\kappa,\mu)$-contact space form with divided $R_5$,
$\~M(f_1, \ldots ,f_6)$ such that $\xi\in\x(M)^\bot$, the scalar curvature and the squared mean curvature satisfy
\begin{align}
2\tau&=n(n-1)f_{1}+2(n-1)f_{4}\tr(h^{T})\nonumber\\
&+f_{5,1}\left\lbrace(\tr(h^{T}))^{2}-\Vert h^{T}\Vert^{2}\right\rbrace -f_{5,2}\left\lbrace \Vert(\phi h)^{T}\Vert^{2}-(\tr(\phi h)^{T})^{2}\right\rbrace \label{eq:17}\\
&+n^{2}\Vert H\Vert^{2}-\Vert\sigma\Vert^{2},\nonumber
\end{align}
where
\begin{gather*}
\Vert Q\Vert^{2}=\sum_{i,j=1}^{n}\langle e_{i},Qe_{j}\rangle ^{2},\quad Q\in \lbrace(\phi h)^{T},h^{T}\rbrace,\\
\quad\Vert\sigma\Vert^{2}=\sum_{i,j=1}^{n}\langle\sigma(e_{i},e_{j}),\sigma(e_{i},e_{j})\rangle
\end{gather*}
and $(\phi h)^{T} X$and $h^{T} X$ are the tangential parts of $\phi h X$ and $hX$ respectively for $X\in \x(M)$.
\end{lemma}
\begin{proof}
We choose a local orthonormal frame $\lbrace e_1,\ldots ,e_n\rbrace$ such that $e_1,\ldots ,e_n$ are tangent to $M$, $e_{n+1}$ is parallel to the mean curvature vector $H$. Then from equation of Gauss \eqref{030}, we have
\begin{align}
K(e_{i}\wedge e_{j})=\~{K}(e_{i}\wedge e_{j})+\sum_{r=n+1}^{2m+1}(\sigma_{ii}^{r}\sigma_{jj}^{r}-(\sigma_{ij}^{r})^{2}).\label{022}
\end{align}
From equation (\ref{022}), we get
\begin{align}
 2\tau&=2\~{\tau}(T_{p}M)+n^{2}\Vert H\Vert^{2}-\Vert\sigma\Vert^{2},\label{001}
\end{align}
since
$$2\~{\tau}(T_{p}M)=\sum_{1\leqslant i\neq j\leqslant n}\~{K}(e_{i}\wedge e_{j})=\sum_{1\leqslant i\neq j\leqslant n}\~{R}(e_{i},e_{j},e_{j},e_{i}),$$
for to get $2\tau $ to be enough that we get $\sum_{1\leqslant i\neq j\leqslant n}\~{R}(e_i,e_j,e_j,e_i)$. From \eqref{028}, we have
\begin{align*}
\sum_{1\leqslant i\neq j\leqslant n}&\~{R}(e_{i},e_{j},e_{j},e_{i})=f_{1}\lbrace\sum_{1\leqslant i\neq j\leqslant n}\langle e_{j},e_{j}\rangle\langle e_{i},e_{i}\rangle -\sum_{1\leqslant i\neq j\leqslant n}\langle e_{i},e_{j}\rangle ^{2}\rbrace\\
&+f_{4}\lbrace\sum_{1\leqslant i\neq j\leqslant n} \langle e_{j},e_{j}\rangle\langle h^{T}e_{i},e_{i}\rangle-\sum_{1\leqslant i\neq j\leqslant n} \langle e_{i},e_{j}\rangle\langle h^{T}e_{j},e_{i}\rangle\\
&+\sum_{1\leqslant i\neq j\leqslant n} \langle h^{T}e_{j},e_{j}\rangle\langle e_{i},e_{i}\rangle-\sum_{1\leqslant i\neq j\leqslant n} \langle h^{T}e_{i},e_{j}\rangle\langle e_{j},e_{i}\rangle\rbrace\\
&+f_{5,1}\lbrace\sum_{1\leqslant i\neq j\leqslant n} \langle h^{T}e_{j},e_{j}\rangle\langle h^{T}e_{i},e_{i}\rangle-\sum_{1\leqslant i\neq j\leqslant n} \langle h^{T}e_{i},e_{j}\rangle\langle h^{T}e_{j},e_{i}\rangle\rbrace\\
&-f_{5,2}\lbrace\sum_{1\leqslant i\neq j\leqslant n} \langle (\phi h)^{T}e_{i},e_{j}\rangle\langle (\phi h)^{T}e_{j},e_{i}\rangle\nonumber\\
&\qquad\qquad-\sum_{1\leqslant i\neq j\leqslant n} \langle (\phi h)^{T}e_{j},e_{j}\rangle\langle (\phi h)^{T}e_{i},e_{i}\rangle\rbrace,
\end{align*}
or
\begin{align*}
2\~{\tau}(T_{p}M)=&\sum_{1\leqslant i\neq j\leqslant n}\~{R}(e_{i},e_{j},e_{j},e_{i})\\
=&n(n-1)f_{1}+2(n-1)f_{4}\tr(h^{T})\\
&+f_{5,1}\lbrace(\tr(h^{T}))^{2}-\sum_{1\leqslant i\neq j\leqslant n}\langle e_{i},h^{T}e_{j}\rangle^{2}\rbrace\\
&-f_{5,2}\lbrace\sum_{1\leqslant i\neq j\leqslant n}\langle e_{i},(\phi h)^{T}e_{j}\rangle^{2}-(\tr(\phi h)^{T})^{2}\rbrace.
\end{align*}
We obtain
\begin{align}\label{002}
\begin{split}
2\~{\tau}(T_{p}M)=&n(n-1)f_{1}+2(n-1)f_{4}\tr(h^{T})\\
&+f_{5,1}\left\lbrace(\tr(h^{T}))^{2}-\|h^{T}\Vert^{2}\right\rbrace\\
 &-f_{5,2}\left\lbrace \Vert(\phi h)^{T}\Vert^{2}-(\tr(\phi h)^{T})^{2}\right\rbrace.
\end{split}
\end{align}
Now if we put (\ref{002}) in (\ref{001}), obtain (\ref{eq:17}).
\end{proof}

\begin{theorem}\label{th:52}
Let $M$ be an $n$-dimensional $(n\geq3)$ $C$-totally real submanifold  in a $(2m+1)$-dimensional generalized $(\kappa,\mu)$-contact space form with divided $R_5$, $\~M(f_1, \ldots ,f_6)$. Then for each point $p\in M$ and each plane section $\pi\subset T_pM,$ we have
\begin{align}\label{eq:21}\begin{split}
\tau-K(\pi)&\leqslant\dfrac{n^{2}(n-2)}{2(n-1)}\Vert H\Vert^{2}+\dfrac{(n+1)(n-2)}{2}f_{1}\\
&+f_{4}\lbrace(n-1)\tr(h^{T})-\tr(h\vert_{\pi})\rbrace\\
&+\frac{1}{2}f_{5,1}\lbrace(\tr(h^{T}))^{2}-\Vert h^{T}\Vert^{2}-2\det (h\vert_{\pi})\rbrace\\
&-\frac{1}{2}f_{5,2}\lbrace\Vert(\phi h)^{T}\Vert^{2}-(\tr(\phi h)^{T})^{2}+2\det ((\phi h)\vert_{\pi})\rbrace.\end{split}
\end{align}
The equality in (\ref{eq:21}) holds at $p\in M$ if and only if there exist an orthonormal
basis $\lbrace e_1,\ldots ,e_n\rbrace$  of $T_p M$ and an orthonormal basis $\lbrace e_{n+1},\ldots,e_{2m+1}=\xi\rbrace$ of $T^{\perp}_p M$
such that ($\mathtt{a}$) $\pi= \Span \lbrace e_1, e_2\rbrace$ and ($\mathtt{b}$) the forms of shape operators $A_r\equiv A_{e_{r}}$, $r=n+1,\ldots,2m+1$, become
\begin{align}\label{eq:22}
A_{n+1}=
\begin{bmatrix}
~~~a&~~~~~~0&0~~~\\
~~~0&~~~~~~b&0~~~\\
~~~0&~~~~~~0&(a+b)I_{n-2}
\end{bmatrix},
\end{align}
\begin{align}\label{eq:23}
A_{r} =
\begin{bmatrix}
~~~c_{r}&~~~~~d_{r}&~~~0~~~~~\\
~~~d_{r}&~~~~~-c_{r}&~~~0~~~~~\\
~~~0&~~~~~0&~~~0_{n-2}
\end{bmatrix},~~~~r=n+2,\ldots,2m+1
\end{align}
\end{theorem}
\begin{proof}
We set
\begin{align}\label{eq:24}
\begin{split}
\rho =2\tau&-n(n-1)f_{1}-2(n-1)f_{4}\tr(h^{T})\\
&-f_{5,1}\left\lbrace(\tr(h^{T}))^{2}-\Vert h^{T}\Vert^{2}\right\rbrace \\
&+f_{5,2}\left\lbrace \Vert(\phi h)^{T}\Vert^{2}-(\tr(\phi h)^{T})^{2}\right\rbrace\\
&-\dfrac{n^{2}(n-2)}{n-1}\Vert H\Vert^{2}.
\end{split}
\end{align}
From (\ref{eq:17}) and (\ref{eq:24}), we get
\begin{align}\label{eq:25}
n^2\Vert H\Vert^2=(n-1)(\Vert \sigma\Vert^2+\rho).
\end{align}
Let $\pi\subset T_pM$ be a plane section. We choose an orthonormal basis $\{e_1,\ldots, e_n\}$ for $T_pM$
and an orthonormal basis $\{e_{n+1},\ldots, e_{2m}, e_{2m+1}\}$ for the normal space $(T_p^\bot M)$
at $p$ such
that $\pi = \Span \{e_1,e_2\}$, the mean curvature vector $H(p)$ is parallel to $e_{n+1}$ and $e_{2m+1} = \xi$ then the
equation (\ref{eq:25}) can be written as
\begin{align}\label{eq:26}\begin{split}
\left(\sum_{i=1}^{n}\sigma_{ii}^{n+1}\right)^2=(n-1)\bigg(\sum_{i=1}^{n}(\sigma_{ii}^{n+1})^2&+\sum_{i\neq j}(\sigma_{ij}^{n+1})^2\\
&+\sum_{r=n+2}^{2m+1}\sum_{i,j=1}^{n}(\sigma_{ii}^{r})^2+\rho\bigg).\end{split}
\end{align}
Applying Lemma \ref{lem:cheninequality}, from (\ref{eq:26}), we obtain
\begin{align}\label{eq:27}
2\sigma_{11}^{n+1}\sigma_{22}^{n+1}\geq \sum_{i\neq j}(\sigma_{ij}^{n+1})^2+\sum_{r=n+2}^{2m+1}\sum_{i,j=1}^{n}(\sigma_{ii}^{r})^2+\rho.
\end{align}
From equation (\ref{030}), we also have
\begin{align}\label{eq:28}\begin{split}
K(\pi)=f_1+f_4\tr(h|_\pi)+f_{5,1}\det(h|_\pi)+f_{5,2}\det((\phi h)|_\pi)\\
+\sigma_{11}^{n+1}\sigma_{22}^{n+1}-(\sigma_{12}^{n+1})^2+\sum_{r=n+2}^{2m+1}(\sigma_{11}^{r}\sigma_{22}^{r}-(\sigma_{12}^{r})^2),
\end{split}
\end{align}
which in view of (\ref{eq:27}) gives
\begin{align}\label{eq:29}
K(\pi)\geq f_1+f_4\tr(h|_\pi)+f_{5,1}\det(h|_\pi)+f_{5,2}\det((\phi h)|_\pi)+\dfrac{1}{2}\rho\nonumber\\
+\sum_{r=n+1}^{2m+1}\sum_{j>2}((\sigma_{1j}^{r})^2+(\sigma_{2j}^{r})^2)+\dfrac{1}{2}\sum_{i\neq j>2}(\sigma_{ij}^{n+1})^2\\
+\dfrac{1}{2}\sum_{r=n+2}^{2m+1}\sum_{i,j>2}(\sigma_{ij}^{r})^2+\sum_{r=n+2}^{2m+1}(\sigma_{11}^r+\sigma_{22}^r)^2,\nonumber
\end{align}
or
\begin{align}\label{eq:30}
K(\pi)\geq f_1+f_4\tr(h|_\pi)+f_{5,1}\det(h|_\pi)+f_{5,2}\det((\phi h)|_\pi)+\dfrac{1}{2}\rho.
\end{align}
In view of (\ref{eq:24}) and (\ref{eq:30}), we obtain (\ref{eq:21}).
If the equality in (\ref{eq:21}) holds, then the inequalities given by (\ref{eq:27}) and (\ref{eq:29}) become
equalities. In this case, we have
\begin{align}\label{eq:31}
\sigma_{1j}^{n+1}=0, \sigma_{2j}^{n+1}=0, \sigma_{ij}^{n+1}=0, i\neq j>2;\nonumber\\
\sigma_{1j}^{r}=\sigma_{2j}^{r}=\sigma_{ij}^{r}=0, r=n+2,\ldots, 2m+1; i,j=3,\ldots,n;\\
\sigma_{11}^{n+2}+\sigma_{22}^{n+2}=\cdots=\sigma_{11}^{2m+1}+\sigma_{22}^{2m+1}=0.\nonumber
\end{align}
Furthermore,
we may choose $e_1$ and $e_2$ so that $\sigma_{12}^{n+1}=0$. Moreover, by applying
Lemma \ref{lem:cheninequality}, we also have
\begin{align}\label{eq:32}
\sigma_{11}^{n+1}+\sigma_{22}^{n+1}=\sigma_{33}^{n+1}=\cdots=\sigma_{nn}^{n+1}.
\end{align}
Thus, after choosing a suitable orthonormal basis, the shape operator of $M$ becomes of
the form given by (\ref{eq:22}) and (\ref{eq:23}). The converse is straightforward.
\end{proof}

\section{Squared mean curvature and Ricci curvature}
Chen established a sharp relationship between the Ricci curvature and the squared
mean curvature for submanifolds in real space forms \cite{Chen:RelationsBetweenRicci}.
We prove similar inequalities for certain submanifolds of a generalized $(\kappa,\mu)$-contact space form with divided $R_5$, $\~M(f_1, \ldots ,f_6)$.
\begin{theorem}\label{th:61}
Let $M$ be an $n$-dimensional $(n\geq3)$ $C$-totally real submanifold  in a $(2m+1)$-dimensional generalized $(\kappa,\mu)$-contact space form with divided $R_5$, $\~M(f_1, \ldots ,f_6)$. Then for each point $p\in M$
\begin{enumerate}
  \item For all unit vector $U\in T_pM,$ we have
  \begin{align}\label{eq:33}
\Ric(U)\leq&\dfrac{1}{4}n^2\Vert H\Vert^2+(n-1)f_1+f_4(\tr(h^T)+(n-2)g(h^TU,U))\nonumber\\
&+f_{5,1}\left(\tr(h^T)g(h^TU,U)-\Vert h^TU\Vert^2\right)\\
&+f_{5,2}\left(\tr((\phi h)^T)g((\phi h)^TU,U)-\Vert (\phi h)^TU\Vert^2\right).\nonumber
  \end{align}
  \item For $H(p)=0,$ a unit tangent vector $U\in T_pM$ satisfies the equality case of \eqref{eq:33} if and only if $U$ belongs to the relative null space $\mathcal{N}_p.$
  \item the equality in \eqref{eq:33}
 holds identically for all unit tangent vectors at $p$ if and only if either $p$ is a totally geodesic point or $n=2$ and $p$ is a totally umbilical point.
  \end{enumerate}
\end{theorem}
\begin{proof}
Let $U\in T_pM$ be a unit tangent vector. We choose an orthonormal basis
$e_1,\ldots, e_n,$ $e_{ n + 1} , \ldots , e_{2m+1}$ such that $e_1,\ldots,e_n$ are tangential to $M$ at $p$ with $e_1 = U$. Then, the squared second fundamental form and the squared mean curvature satisfy the
following relation
\begin{align}\begin{split}
\Vert \sigma \Vert^2=&\frac{1}{2}n^2\Vert H\Vert^2+\frac{1}{2}\sum_{r=n+1}^{2m+1}(\sigma^r_{11}-\sigma_{22}^r\cdots-\sigma^r_{nn})^2\label{eq:34}\\
&+2\sum_{r=n+1}^{2m+1}\sum_{j=2}^{n}(\sigma^r_{1j})^2-2\sum_{r=n+1}^{2m+1}\sum_{2\leq i< j\leq n}(\sigma_{ii}^r\sigma_{jj}^r-(\sigma_{ij}^r)^2).
\end{split}
\end{align}
From (\ref{eq:17}) and (\ref{eq:34}), we get
\begin{align}
\tau&-\dfrac{n(n-1)}{2}f_{1}-(n-1)f_{4}\tr(h^{T})\nonumber\\
&-\dfrac{1}{2}f_{5,1}\left\lbrace(\tr(h^{T}))^{2}-\Vert h^{T}\Vert^{2}\right\rbrace +\dfrac{1}{2}f_{5,2}\left\lbrace \Vert(\phi h)^{T}\Vert^{2}-(\tr(\phi h)^{T})^{2}\right\rbrace \nonumber\\
&+\frac{1}{4}\sum_{r=n+1}^{2m+1}(\sigma^r_{11}-\sigma_{22}^r\cdots-\sigma^r_{nn})^2\label{eq:35}\\
&+\sum_{r=n+1}^{2m+1}\sum_{j=2}^{n}(\sigma^r_{1j})^2-\sum_{r=n+1}^{2m+1}\sum_{2\leq i<j \leq n}(\sigma_{ii}^r\sigma_{jj}^r-(\sigma_{ij}^r)^2)\nonumber\\
&=\dfrac{1}{4}n^{2}\Vert H\Vert^{2}.\nonumber
\end{align}
From \eqref{030} and (\ref{028}), we also have
\begin{align}\label{eq:36}\begin{split}
\sum_{2\leq i< j\leq n}K_{ij}=&\sum_{r=n+1}^{2m+1}\sum_{2\leq i< j\leq n}(\sigma_{ii}^r\sigma_{jj}^r-(\sigma_{ij}^r)^2)+\dfrac{(n-1)(n-2)}{2}f_{1}\\
&+\dfrac{1}{2}f_{5,1}\bigg\lbrace(\tr(h^{T}))^{2}-2\tr(h^T)\langle h^T e_1,e_1\rangle\\
&\hspace{1cm}-\|h^{T}\Vert^{2}+2\Vert h^T e_1\Vert^2\bigg\rbrace \\
&-\dfrac{1}{2}f_{5,2}\bigg\lbrace \Vert(\phi h)^{T}\Vert^{2}-(\tr(\phi h)^{T})^{2}-2\Vert(\phi h)^T e_1\Vert^2\\
&\qquad\qquad+2\tr((\phi h)^T)\langle(\phi h)^T e_1,e_1\rangle\bigg\rbrace\\
&+f_4(n-2)(\tr(h^T)-\langle h^Te_1,e_1\rangle).\end{split}
\end{align}
From (\ref{eq:35}) and (\ref{eq:36}), we get
\begin{align}
\Ric(U)=&\dfrac{1}{4}n^2\Vert H\Vert^2+(n-1)f_1+f_4\tr(h^T)+(n-2)f_4\langle h^TU,U\rangle\nonumber\\
&+f_{5,1}\big(\tr(h^T)\langle h^TU,U\rangle-\Vert h^TU\Vert^2\big)\nonumber\\
&+f_{5,2}\big(\tr((\phi h)^T)\langle(\phi h)^TU,U\rangle-\Vert (\phi h)^TU\Vert^2\big)\nonumber\\
&-\frac{1}{4}\sum_{r=n+1}^{2m+1}(\sigma^r_{11}-\sigma_{22}^r\cdots-\sigma^r_{nn})^2\label{eq:37}\\
&-\sum_{r=n+1}^{2m+1}\sum_{j=2}^{n}(\sigma^r_{1j})^2,\nonumber
\end{align}
which implies (\ref{eq:33}).

Assuming $U=e_1$, from (\ref{eq:37}), the equality case of (\ref{eq:33}) is valid if and only if
\begin{align}
\sigma_{11}^r&=\sigma_{22}^r+\cdots+\sigma_{nn}^r,\label{eq:38}\\
\sigma_{12}^r&=\cdots=\sigma_{1n}^r=0, r=n+1, \ldots,2m+1.\nonumber
\end{align}
If $H(p) = 0,$ (\ref{eq:38}) implies that $U = e_1$ lies in the relative null space $\mathcal{N}_p$. Conversely, if
$U=e_1$ lies in the relative null space $\mathcal{N}_p$, then (\ref{eq:38}) holds, since $H(p) = 0$ is assumed.
Thus (2) is proved.

Now we prove (3). The equality case of \eqref{eq:33} for all unit tangent vectors to $M$ at $p$
happens if and only if
\begin{align}
2\sigma_{ii}^r&=\sigma_{11}^r+\cdots+\sigma_{nn}^r, i=1,\ldots, n, r=n+1,\ldots,2m+1,\label{eq:39}\\
\sigma_{ij}^r&=0, i\neq j,  r=n+1,\ldots, 2m+1.\nonumber
\end{align}
Thus, we have two cases, namely either $n = 2$ or $n\neq2$. In the first case $p$ is a totally
umbilical point, while in the second case $p$ is a totally geodesic point. The proof of
converse part is straightforward.
\end{proof}
\section{Squared mean curvature and $k$-Ricci curvature}
\begin{theorem}
Let $M$ be an $n$-dimensional $(n\geq3)$ $C$-totally real submanifold  in a $(2m+1)$-dimensional generalized $(\kappa,\mu)$-contact space form with divided $R_5$, $\~M(f_1,\ldots,f_6)$. Then we have
 \begin{align}\label{eq:40}\begin{split}
n(n-1)\Vert H\Vert^{2}\geq&2\tau -n(n-1)f_{1}-2(n-1)f_{4}\tr(h^{T})\\
&-f_{5,1}\left\lbrace(\tr(h^{T}))^{2}-\Vert h^{T}\Vert^{2}\right\rbrace \\
&+f_{5,2}\left\lbrace \Vert(\phi h)^{T}\Vert^{2}-(\tr(\phi h)^{T})^{2}\right.\rbrace\end{split}
\end{align}

\end{theorem}

\begin{proof}
Let $X \in T_pM$ be a unit tangent vector $X$ at $p$. We choose an orthonormal
basis $\{e_1,\ldots, e_n,e_{n+1},\ldots, e_{2m+1}\}$ such that $e_1,\ldots, e_n$ are tangential to $M$ at $p$ with
$e_1 = X$. We recall the equation (\ref{eq:17}) as
\begin{align}
n^{2}\Vert H\Vert^{2}&=2\tau +\Vert\sigma\Vert^{2} -n(n-1)f_{1}-2(n-1)f_{4}\tr(h^{T})\nonumber\\
&-f_{5,1}\left\lbrace(\tr(h^{T}))^{2}-\Vert h^{T}\Vert^{2}\right\rbrace +f_{5,2}\left\lbrace \Vert(\phi h)^{T}\Vert^{2}-(\tr(\phi h)^{T})^{2}\right\rbrace. \label{eq:41}
\end{align}
Let the orthonormal basis $\{e_1,\ldots, e_n, e_{n+1},\ldots, e_{2m+1}\}$ be
such that $e_{n+1}$ is parallel to the mean curvature vector $H(p)$ and $e_1,\ldots, e_n$ diagonalize
the shape operator $A_{n+1}$. Then the shape operators take the forms
\begin{align}
A_{n+1}=&\begin{bmatrix}
a_1&0&\cdots&0\\
0&a_2&\cdots&0\\
\vdots&\vdots&\ddots&\vdots\\
0&0&\cdots&a_n
\end{bmatrix},\label{eq:42}\\
A_r=(\sigma_{ij}^r), i,j=1,\ldots,n , r=&n+2, \ldots, 2m+1, \tr(A_r)=\sum_{i=1}^{n}\sigma_{ii}^r=0.\label{eq:43}
\end{align}
From (\ref{eq:41}), we get
\begin{align}
n^{2}\Vert H\Vert^{2}&=2\tau +\sum_{i=1}^{n}a_i^2 +\sum_{r=n+2}^{2m+1}\sum_{i,j=1}^{n}(\sigma_{ij}^r)^2-n(n-1)f_{1}-2(n-1)f_{4}\tr(h^{T})\nonumber\\
&-f_{5,1}\left\lbrace(\tr(h^{T}))^{2}-\Vert h^{T}\Vert^{2}\right\rbrace +f_{5,2}\left\lbrace \Vert(\phi h)^{T}\Vert^{2}-(\tr(\phi h)^{T})^{2}\right\rbrace.\label{eq:44}
\end{align}
Since
\begin{align}
0\leq\sum_{i<j}(a_i-a_j)^2=(n-1)\sum_ia_i^2-2\sum_{i<j}a_ia_j.
\end{align}
Therefor, we get
\begin{align}
n^2\Vert H\Vert^2=\left(\sum_{i=1}^{n}a_i\right)^2=\sum_{i=1}^{n}a_i^2+2\sum_{i<j}a_ia_j\leq n\sum_{i=1}^{n}a_i^2,\label{eq:45}
\end{align}
which implies
$$\sum_{i=1}^{n}a_i^2\geq n\Vert H\Vert^2.$$
In view of (\ref{eq:44}), we obtain
\begin{align}
n^{2}\Vert H\Vert^{2}&\geq2\tau +n \Vert H\Vert^2-n(n-1)f_{1}-2(n-1)f_{4}\tr(h^{T})\nonumber\\
&-f_{5,1}\left\lbrace(\tr(h^{T}))^{2}-\Vert h^{T}\Vert^{2}\right\rbrace +f_{5,2}\left\lbrace \Vert(\phi h)^{T}\Vert^{2}-(\tr(\phi h)^{T})^{2}\right\rbrace,\label{eq:46}
\end{align}
which gives (\ref{eq:40}).
\end{proof}
\begin{theorem}
Let $M$ be an $n$-dimensional $(n\geq3)$ $C$-totally real submanifold  in a $(2m+1)$-dimensional generalized $(\kappa,\mu)$-contact space form with divided $R_5$, $\~M(f_1, \ldots ,f_6)$. Then we have
 \begin{align}\label{eq:47}
n(n-1)\Vert H\Vert^{2}&\geq n(n-1)\theta_k(p) -n(n-1)f_{1}-2(n-1)f_{4}\tr(h^{T})\nonumber\\
&-f_{5,1}\left\lbrace(\tr(h^{T}))^{2}-\Vert h^{T}\Vert^{2}\right\rbrace +f_{5,2}\left\lbrace \Vert(\phi h)^{T}\Vert^{2}-(\tr(\phi h)^{T})^{2}\right.\rbrace
\end{align}
\end{theorem}
\begin{proof}
Let $\{e_1,\ldots,e_n\}$ be an orthonormal basis of $T_pM.$ We denote by $L_{i_1\ldots i_k}$ the
$k$-plane section spanned by $e_{i_1},\ldots, e_{i_k}$. From (\ref{eq:11}) and (\ref{eq:10}), it follows that
\begin{align}
\tau(L_{i_1\ldots i_k})=\dfrac{1}{2}\sum_{i\in\{{i_1},\ldots, {i_k}\}}\Ric_{L_{i_1\ldots i_k}}(e_i),\label{eq:48}
\end{align}
and
\begin{align}\label{eq:49}
\tau(p)=\dfrac{1}{C^{n-2}_{k-2}}\sum_{1\leq i_1<\cdots<i_k\leq n}\tau(L_{i_1\ldots i_k}).
\end{align}
Combining (\ref{eq:12}), (\ref{eq:48}) and (\ref{eq:49}), we obtain
\begin{align}
\tau(p)\geq\dfrac{n(n-1)}{2}\theta_k(p),
\end{align}
which in view of (\ref{eq:40}) implies (\ref{eq:47}).
\end{proof}

\section{$\delta$-invariant and inequalities for submanifolds in a Sasakian  generalized $(\kappa,\mu)$-space forms}
In this section,we apply these results to get corresponding results for $C$-totally real submanifolds in a generalized $(\kappa,\mu)$-contact space forms $\~M(f_1, \ldots ,f_6)$ with $f_3=f_1-1$. If $\kappa = 1=f_1-f_3$, the  generalized $(\kappa,\mu)$-contact space forms $\~M(f_1, \ldots ,f_6)$
reduces to Sasakian space form; thus $h = 0$ and (\ref{028}) becomes
\begin{align}\label{eq:51}
\~R(X,Y)Z=f_1R_1(X,Y)Z+f_2R_2(X,Y)Z+f_3R_3(X,Y)Z.
\end{align}
Moreover, for a $C$-totally real submanifolds in Sasakian space forms, from (\ref{4.1}), we also
get
\begin{align}
A_\xi=0.\label{eq:52}
\end{align}
Thus, in view of Theorem \ref{th:52}, we can state the following.

\begin{theorem}\label{th:81}
Let $M$ be an $n$-dimensional $(n\geq3)$ $C$-totally real submanifold  in a $(2m+1)$-dimensional generalized $(\kappa,\mu)$-contact space forms $\~M(f_1, \ldots ,f_6)$ satisfying $f_3=f_1-1$.
We have
\begin{align}
\delta_M&\leqslant\dfrac{n^{2}(n-2)}{2(n-1)}\Vert H\Vert^{2}+\dfrac{(n+1)(n-2)}{2}f_{1}.\label{eq:53}
\end{align}
The equality in (\ref{eq:53}) holds at $p\in M$ if and only if there exist an orthonormal
basis $\lbrace e_1,\ldots ,e_n\rbrace$  of $T_p M$ and an orthonormal basis $\lbrace e_{n+1},\ldots,e_{2m+1}=\xi\rbrace$ of $T^{\perp}_p M$
such that ($\mathtt{a}$) $\pi= \Span \lbrace e_1, e_2\rbrace$ and ($\mathtt{b}$) the forms of shape operators $A_r\equiv A_{e_{r}}$ , $r=n+1,\ldots,2m+1$, become
\begin{align}\label{eq:54}
A_{n+1}=
\begin{bmatrix}
~~~a&~~~~~~0&0~~~\\
~~~0&~~~~~~b&0~~~\\
~~~0&~~~~~~0&(a+b)I_{n-2}
\end{bmatrix},
\end{align}
\begin{align}\label{eq:55}
A_{r} =
\begin{bmatrix}
~~~c_{r}&~~~~~d_{r}&~~~0~~~~~\\
~~~d_{r}&~~~~~-c_{r}&~~~0~~~~~\\
~~~0&~~~~~0&~~~0_{n-2}
\end{bmatrix},~~~~r=n+2,\ldots,2m,
\end{align}
and $A_\xi=0.$
\end{theorem}
\begin{theorem}\label{th:82}
Let $M$ be an $n$-dimensional $(n\geq3)$ $C$-totally real submanifold  in a $(2m+1)$-dimensional generalized $(\kappa,\mu)$-contact space forms $\~M(f_1, \ldots ,f_6)$ satisfying $f_3=f_1-1$. Then for each point $p\in M$
\begin{enumerate}
  \item For all unit vector $U\in T_pM,$ we have
  \begin{align}
\Ric(U)\leq&\dfrac{1}{4}n^2\Vert H\Vert^2+(n-1)f_1.\label{eq:56}
  \end{align}
  \item For $H(p)=0,$ a unit tangent vector $U\in T_pM$ satisfies
  \begin{align}\label{eq:57}
  \Ric(U)=(n-1)f_1
  \end{align}
   if and only if $U$ belongs to the relative null space $\mathcal{N}_p.$
  \item For each $p \in M$
  \begin{align}
  4S\leq&(n^2\Vert H\Vert^2+4(n-1)f_1)g.\label{eq:58}
  \end{align}
  \end{enumerate}
  where $S$ is the Ricci tensor of the submanifold and $g$ is the Riemannian metric. The equality in (\ref{eq:58}) holds
if and only if either $p$ is a totally geodesic point or $n= 2$ and $p$ is a totally
umbilical point.
\end{theorem}

\begin{theorem}
Let $M$ be an $n$-dimensional $(n\geq3)$ $C$-totally real submanifold  in a $(2m+1)$-dimensional generalized $(\kappa,\mu)$-contact space form with divided $R_5$, $\~M(f_1, \ldots ,f_6)$ satisfying $f_3=f_1-1$. Then we have
\begin{align}\label{eq:311}
\delta \left(
n_{1},\ldots ,n_{k}\right)\leq c\nn\Vert H\Vert^2+b\nn f_1
\end{align}
for any $k$-tuple $\nn\in\mathcal{S}(n).$ The equality case of inequality \eqref{eq:311} holds at a point $p\in M$ if and only if there exists an orthonormal basis $\{e_1,\ldots,e_{2m+1}\}$ at $p$ such that the shape operators of $M$ in $\~M$ at $p$ take the following form:
\begin{align}\label{eq:312}
A_r=\begin{pmatrix}
A^r_1&\cdots&0&\\
\vdots&\ddots&\vdots&0\\
0&\cdots&A^r_k&\\
&0&&a_rI
\end{pmatrix}, r=n+1,\ldots,2m+1,
\end{align}
where $I$ is an identity matrix and $A^r_i$ are symmetric $n_i\times n_i$ submatrices such that
\begin{align}\label{eq:313}
\tr(A^r_1)=\ldots=\tr(A^r_k)=a_r.
\end{align}
\end{theorem}
\begin{proof}
Let $\n\in\mathcal{S}(n).$ We set
\begin{align}\label{eq:epsilon}\begin{split}
\epsilon=2\tau-2c\n\Vert H\Vert^2-n(n-1)f_1.
\end{split}
\end{align}
Substituting \eqref{eq:17} in \eqref{eq:epsilon}, we have
\begin{align}\label{eq:315}
n^2\Vert H\Vert^2=\gamma(\epsilon+\Vert \sigma\Vert^2),\qquad \gamma=n+k-\sum_{j=1}^kn_j.
\end{align}
Let $L_1, \ldots, L_k$ be mutually orthogonal subspaces of $T_pM$ with $\dim L_j=n_j, j=1,\ldots,k.$ By choosing an orthonormal basis $\{e_1,\ldots,e_n,\ldots,e_{2m+1}\}$ at $p\in M$ such that $e_1,\ldots,e_n$ are tangent to $M$ at $p$, $e_{n+1}=\frac{H}{\Vert H\Vert}, e_{2m+1}=\xi$, and
\begin{align*}
L_1=&\Span\{e_1,\ldots,e_{n_1}\},\\
L_2=&\Span\{e_{n_1+1},\ldots,e_{n_1+n_2}\},\\
\vdots &\\
L_k=&\Span\{e_{n_1+\cdots+n_{k-1}+1},\ldots,e_{n_1+\cdots+n_k}\},
\end{align*}
then the equation (\ref{eq:315}) can be written as
\begin{align}\label{eq:316}\begin{split}
\left(\sum_{i=1}^{n}\sigma_{ii}^{n+1}\right)^2=\gamma\bigg(\epsilon+\sum_{i=1}^{n}(\sigma_{ii}^{n+1})^2&+\sum_{i\neq j}(\sigma_{ij}^{n+1})^2
+\sum_{r=n+2}^{2m+1}\sum_{i,j=1}^{n}(\sigma_{ii}^{r})^2\bigg).\end{split}
\end{align}
We set
\[\Delta_1=\{1,\ldots,n_1\},\ldots,\Delta_k=\{n_1+\cdots+n_{k-1}+1,\ldots,n_1+\cdots+n_k\}.\]
The equation \eqref{eq:316} can be rewritten in the form
\begin{align}\label{eq:317}\begin{split}
&\left(\sum_{i=1}^{\gamma+1}a_i\right)^2=\gamma\bigg(\epsilon+\sum_{i=1}^{\gamma+1}(a_i)^2+\sum_{i\neq j}(\sigma_{ij}^{n+1})^2+\sum_{r=n+2}^{2m+1}\sum_{i,j=1}^{n}(\sigma_{ii}^{r})^2\\
&-\hspace{-5mm}\sum_{2\leq \alpha_1\neq\beta_1\leq n_1}\hspace{-5mm}\sigma_{\alpha_1\alpha_1}^{n+1}\sigma_{\beta_1\beta_1}^{n+1}-\sum_{ \alpha_2\neq\beta_2}\sigma_{\alpha_2\alpha_2}^{n+1}\sigma_{\beta_2\beta_2}^{n+1}-\cdots-\sum_{\alpha_k\neq\beta_k}\sigma_{\alpha_k\alpha_k}^{n+1}\sigma_{\beta_k\beta_k}^{n+1}\bigg),\end{split}
\end{align}
where $\alpha_i, \beta_i\in\Delta_i, i=2, \ldots, k$ and
\begin{align*}
a_1&=\sigma_{11}^{n+1},
a_2=\sum_{i=2}^{n_1}\sigma_{ii}^{n+1},
a_3=\sum_{i=n_1+1}^{n_1+n_2}\sigma_{ii}^{n+1},\ldots,
a_{k+1}=\hspace{-5mm}\sum_{i=n_1+\cdots+n_{k-1}+1}^{n_1+\cdots+n_k}\hspace*{-5mm}\sigma_{ii}^{n+1},\\
a_{k+2}&=\sigma_{(n_1+\cdots+n_k+1)(n_1+\cdots+n_k+1)}^{n+1},\ldots,a_{\gamma+1}=\sigma_{nn}^{n+1}.
\end{align*}
By applying Lemma \ref{lem:cheninequality} to \eqref{eq:317}, we can obtain the following inequality
\begin{align}\label{eq:318}\begin{split}
\sum_{\alpha_1<\beta_1}\sigma_{\alpha_1\alpha_1}^{n+1}\sigma_{\beta_1\beta_1}^{n+1}+\sum_{ \alpha_2<\beta_2}\sigma_{\alpha_2\alpha_2}^{n+1}\sigma_{\beta_2\beta_2}^{n+1}+\cdots+\sum_{\alpha_k<\beta_k}\sigma_{\alpha_k\alpha_k}^{n+1}\sigma_{\beta_k\beta_k}^{n+1}\\
\geq\frac{\epsilon}{2}+\sum_{i<j}(\sigma_{ij}^{n+1})^2+\frac{1}{2}\sum_{r=n+2}^{2m+1}\sum_{i,j=1}^{n}(\sigma_{ij}^{r})^2,
\end{split}
\end{align}
where $\alpha_i, \beta_i\in\Delta_i, i=1, \ldots, k.$
From (\ref{eq:11}) and Gauss'equation, we see that
\begin{align}\label{eq:319}
\tau(L_j)=\dfrac{n_j(n_j-1)}{2}f_1+\sum_{r=n+1}^{2m+1}\sum_{\alpha_i<\beta_i}(\sigma_{\alpha_i\alpha_i}^{r}\sigma_{\beta_i\beta_i}^{r}-(\sigma_{\alpha_i\beta_i}^{r})^2),
\end{align}
where $\alpha_i, \beta_i\in\Delta_i, i=1, \ldots, k.$ In view of \eqref{eq:318} and \eqref{eq:319}, we obtain
\begin{align} \label{eq:320}\begin{split}
\sum_{i=1}^{k}\tau(L_i)\geq&\frac{\epsilon}{2}+\sum_{i=1}^{k}\dfrac{n_i(n_i-1)}{2}f_1\\
&+\frac{1}{2}\sum_{r=n+1}^{2m+1}\sum_{(\alpha,\beta)\notin\Delta^2}(\sigma_{\alpha\beta}^r)^2
+\frac{1}{2}\sum_{r=n+2}^{2m+1}\sum_{i=1}^{k}\bigg(\sum_{\alpha_i\in\Delta_i}\sigma_{\alpha_i\alpha_i}^r\bigg)^2\\
\geq&\frac{\epsilon}{2}+\sum_{i=1}^{k}\dfrac{n_i(n_i-1)}{2}f_1,
\end{split}
\end{align}
where $\Delta^2=(\Delta_1\times\Delta_1)\cup\ldots\cup(\Delta_k\times\Delta_k).$ From \eqref{eq:deltainf}, \eqref{eq:epsilon} and  \eqref{eq:320}, we can obtain \eqref{eq:311}. If the equality in \eqref{eq:311} holds at a point $p$, then the inequalities in \eqref{eq:318} and \eqref{eq:320} are actually equalities at $p$. In this case, by applying Lemma \ref{lem:cheninequality}, \eqref{eq:317}, \eqref{eq:318}, \eqref{eq:319} and \eqref{eq:320}, we also obtain \eqref{eq:312} and \eqref{eq:313}.
\end{proof}


\end{document}